\documentclass[12pt]{article}
\usepackage{amssymb,amsmath,graphicx,amsthm,mathrsfs}
\usepackage{color}
\newtheorem{Theorem}{Theorem}[section]
\newtheorem{Lemma}[Theorem]{Lemma}
\newtheorem{Proposition}[Theorem]{Proposition}

\newtheorem{Corollary}[Theorem]{Corollary}

\newtheorem{Remark}[Theorem]{Remark}
\numberwithin{equation}{section}

\newcommand{\lc}
{\mathrel{\raise2pt\hbox{${\mathop<\limits_{\raise1pt\hbox
{\mbox{$\sim$}}}}$}}}

\newcommand{\gc}
{\mathrel{\raise2pt\hbox{${\mathop>\limits_{\raise1pt\hbox{\mbox{$\sim$}}}}$}}}

\newcommand{\ec}
{\mathrel{\raise2pt\hbox{${\mathop=\limits_{\raise1pt\hbox{\mbox{$\sim$}}}}$}}}

\def \lg {\langle}
\def \rg {\rangle}
\def \R {\mathbb R}

\begin{document}

\title{\textbf{Observability inequalities from measurable sets   for some evolution equations}}

\author{
Gengsheng Wang
\thanks{
School of Mathematics and Statistics, Wuhan University, Wuhan,
430072, China (wanggs62@yeah.net). The authors are partially supported by the National Natural Science Foundation of China under grants
11161130003 and 11171264 and by the National Basis Research Program of China (973 Program) under grant 2011CB808002.
}
\quad \quad
Can Zhang\thanks{ School of Mathematics and Statistics, Wuhan
University, Wuhan, 430072, China (zhangcansx@163.com). }
}

\date{}
\maketitle
\begin{abstract}
In this paper, we build up two  observability inequalities
from measurable sets in time for some evolution equations in Hilbert spaces
 from two different settings. The equation reads:  $u'=Au,\; t>0$, and the observation operator is denoted by $B$.
 In the first setting, we assume that $A$ generates an analytic semigroup,
 $B$ is an admissible observation operator for this semigroup (cf. \cite{TG}), and the pair $(A,B)$ verifies some observability inequality from time intervals.
 With the help of the  propagation estimate of analytic functions
 (cf. \cite{V}) and  a telescoping series method provided in the current paper,
  we establish an observability inequality
from measurable sets in time. In the second setting, we suppose that
$A$ generates a $C_0$ semigroup, $B$ is a linear and bounded
operator, and the pair $(A, B)$ verifies some spectral-like
condition. With the aid of methods developed in \cite{AEWZ} and
\cite{PW2} respectively, we first obtain an interpolation inequality
at one time,
 and  then derive an  observability inequality from measurable sets in time.
 These two observability inequalities are applied to get the bang-bang property for some time optimal control problems.

\bigskip

\noindent\textbf{Key words.} Evolution equations in Hilbert spaces,
observability inequality in measurable sets,  telescoping series
method, propagation estimate of analytic functions, bang-bang
property of time optimal controls
\bigskip

\noindent\textbf{AMS Subject Classifications. 93B07, 93C25}
\end{abstract}



\section{Introduction and main results}
The aim of this study is to present an observability inequality from
measurable sets in time for some parabolic-like evolution equations.
Such an  estimate was built up for the heat equation in
\cite{gengshengwang1} and was established for heat equations with
lower order terms depending on both space and time variables $x$ and $t$
in \cite{PW2}.  To the best of our  knowledge,  it has not been
touched upon for abstract evolution equations so far.

We start with introducing  the evolution equation under study:
\begin{equation}\label{a-eq1}
\frac{du}{dt}=Au, \;\;t>0,\;\;\;u(0)=u_0\in X,
\end{equation}
where $X$ is a Hilbert space and $A:D(A)\subset X\rightarrow X$ is
the infinitesimal generator of a $C_0$ semigroup $\{S(t);t\geq0\}$
in $X$. Denote by  $\lg \cdot,\cdot\rg_X$ and $\|\cdot\|_X$  the
inner product and the norm of $X$ respectively, and endow the space
$D(A)$ with the graph norm.

We next introduce an  observation operator $B: X\rightarrow U$ from
two cases. Here $U$ is another Hilbert space with the inner product
$\lg \cdot,\cdot\rg_U $ and the norm  $\|\cdot\|_U$. For each Banach
space $Z$, $\mathcal{L}(Z,U)$  stands for the space of all linear
bounded operators from $Z$ to $U$, with the usual norm
$\|\cdot\|_{\mathcal{L}(Z,U)}$. In the first case, we let
$B\in\mathcal L(D(A),U)$ hold the following two properties:

\noindent $(a)$ $B$ is  an admissible observation operator for $\{S(t);t\geq0\}$, i.e., for each $\tau>0$,
there exists a positive constant $C(\tau)$ such that
\begin{equation}\label{admi}
\int_0^\tau\|BS(t)u_0\|_U^2\,dt\leq C(\tau)\|u_0\|_X^2\;\;\text{for all}\;
\;u_0\in D(A).
\end{equation}

\noindent $(b)$ The pair $(A,B)$ verifies the
 observability inequality from time intervals:
{\it There are two positive constants $d$ and $k$ such that for
 any $L\in(0,1]$,
\begin{equation}\label{wang1.3}
\|S(L)u_0\|_X^2 \leq
e^{\frac{d}{L^k}}\int_{0}^{L}\|BS(t)u_0\|_U^2\,dt\;\;\text{for all}\;\;
u_0\in D(A).
\end{equation}}
Here and throughout this paper,  $C(\cdots)$ denotes a positive
constant depending on what are inclosed in the brackets, and may
vary in different contexts.
 Our definition of admissible
observation operators is quoted  from \cite[Chapter 4]{TG}.
  For more  details on the
above-mentioned inequality (\ref{wang1.3}), we refer the
 readers to \cite[Chapter 2]{Coron} or
\cite[Chapter 6]{TG}. In the second case, we let $B\in\mathcal
L(X,U)$ be such that  the pair $(A,B)$ verifies the  Hypothesis
$(H)$: {\it There is a family of increasing subspaces $\{\mathbb
E_{\lambda_m}\}_{m\geq1}$ of $X$,  with
$$
0<\lambda_1\leq\lambda_2\leq\cdots\leq\lambda_m\rightarrow+\infty,
$$
 verifying \\

\noindent$(i)$ for each $m\in \mathbb{N}$,  $S(t)\mathbb E_{\lambda_m}\subset \mathbb E_{\lambda_m}$ for all $t\geq 0$;\\

\noindent$(ii)$  there is a  constant $\mu>0$ such that for each
$m\in \mathbb{N}$,
$$
\|S(t)g\|_X\leq e^{-\mu t{\lambda_m}}\|g\|_X\;\;\mbox{for all}\;\;
g\in \mathbb E_{\lambda_m}^\bot\;\;\mbox{and}\;\; t>0;
$$

\noindent$(iii)$  there are constants $\gamma\in(0,1)$ and $N\geq1$
such that for each $m\in \mathbb{N}$,
$$
\|f\|_X\leq Ne^{N\lambda_m^\gamma}\|Bf\|_{U}\;\;\mbox{for all}\;\;
f\in\mathbb E_{\lambda_m}.
$$ }
Here, $\mathbb E_{\lambda_m}^\bot$ is the orthogonal complementary
subspace to $\mathbb E_{\lambda_m}$ in $X$.
 We refer to \cite{Seidman-2008} or \cite{TT1} for a similar hypothesis condition to $(H)$.

The main results of this paper are included in the following two
theorems.

\begin{Theorem}\label{main}
Let $A$ generate an analytic semigroup $\{S(t);t\geq0\}$ in $X$ and $B\in\mathcal L(D(A),U)$ verify the admissible observation condition \eqref{admi}.
 Assume that $(A,B)$ holds the observability inequality \eqref{wang1.3}.
 Then, given  $T>0$ and a subset $E\subset(0,T)$ of positive  measure,
 there exists a positive constant $C=C(E,T,d,k, \|B\|_{\mathcal L(D(A),U)})$
such that
\begin{equation}\label{403c8}
\|S(T)u_0\|_X\leq C\int_E\|BS(t)u_0\|_U\,dt\;\;\;\mbox{for all}\;\;
u_0\in D(A).
\end{equation}
\end{Theorem}

\begin{Theorem}\label{bing1} Let $A$ generate a $C_0$ semigroup $\{S(t);t\geq0\}$ in $X$ and $B\in\mathcal
L(X,U)$. Assume that $(A,B)$ satisfies the Hypothesis $(H)$. Then
the following estimates hold:

\noindent$(\textrm{I})$\;There exists a constant $C=C(N,\mu,
\|B\|_{\mathcal L(X,U)},\lambda_1)\geq1$ such that when $t\in(0,1]$,
\begin{equation}\label{hong1}
\|S(t)u_0\|_X\leq
\Big(C\exp\big(Ct^{-\frac{\gamma}{1-\gamma}}\big)\|BS(t)u_0
\|_U\Big)^{\frac{1}{2}}\|u_0\|_X^{\frac12}\;\;\mbox{for all}\;\;
u_0\in X.
\end{equation}

\noindent$(\textrm{II})$\; Given $T>0$ and a subset $E\subset(0,T)$
 of positive measure, there is a  constant
$C=C(E, T, N,\mu, \|B\|_{\mathcal L(X,U)},\lambda_1,\gamma)$ such
that
\begin{equation}\label{hong2}
\|S(T)u_0\|_X\leq C\int_E\|BS(t)u_0\|_{U}\,dt\;\;\mbox{for all}\;\;
u_0\in X.
\end{equation}
\end{Theorem}

 Several remarks are given in order:

 $(1)$  Theorem~\ref{main} can be
applied to get the null controllability from measurable  sets in
time for several important equations: the internally controlled
Stokes equations; the internally controlled degenerate parabolic
equations associated with the Grushin operator in dimension 2; the
boundary controlled heat equations, and so on. More importantly,
with the aid of Theorem~\ref{main}, we can build up the bang-bang
property of time optimal control problems for the above-mentioned
controlled equations. This property is extremely important in the
studies of time optimal control problems (cf., e.g.,  \cite{LW1}, \cite{LW2}, \cite{phungwangzhang}, \cite{PWZ}, \cite{WX},
\cite{WZ}, \cite{Yu}). These applications will be presented in
Section 3 of this paper. It is worth mentioning that for the first
two equations above-mentioned, the corresponding observability
inequality  (\ref{wang1.3}) was built up in \cite{CG} and \cite{BCG}
respectively; while for the last equation, it was provided in
\cite{TG}.

\vskip 5pt
 $(2)$ The inequality (\ref{hong1}) is a quantitative
unique continuation estimate at one time, while the inequality
(\ref{hong2}) is an observability estimate from measurable sets in
time. They have been studied for  heat equations with lower order terms
depending on both space and time variables $x$ and $t$ in
\cite{PW1}, \cite{PW2} and \cite{PWZ}. We derive the estimate
(\ref{hong2}) from the inequality (\ref{hong1}), via the method
provided in \cite{PW2}. In the case where $U=X$ and $B=I$ (the
identity operator on $X$), one can directly check that
$$
\|S(t)u_0\|_X\leq
(C\|S(t)u_0\|_X)^{\frac{1}{2}}\|u_0\|_X^{\frac{1}{2}}\;\;\mbox{for
all}\;\; u_0\in X, \; t\in(0, 1],
$$
which leads to (\ref{hong1}). Consequently, (\ref{hong2}) holds.
Hence, the assumption $(H)$ is not necessary in this case. From
(\ref{hong2}), the bang-bang property for the corresponding time
optimal control problem  follows.  Such property for this special
case was first established in \cite{Fa1964} by a different way.

\vskip 5pt

 $(3)$ Consider the more general evolution equation:
$$
\frac{du}{dt}=A(t)u,\; t>0,\;\;u(0)=u_0,
$$
where $A(\cdot)$ verifies certain conditions such that the above
equation is well-posed and the solution is analytic in time  (cf.
\cite[Part 3, Theorem 2.2]{Friedman}, \cite[Chapter 5]{Pa}). It
seems for us  that one can  get a similar estimate to (\ref{403c8})
for the aforementioned time-varying equation, through utilizing a
similar method  to that used in the proof of Theorem~\ref{main}.

\vskip 5pt

$(4)$ We call the  inequality  (\ref{wang1.3})  an
$L^2$-observability inequality  on time intervals, since the
integral on its right hand side is the $L^2(0,T; U)$-norm of
$BS(\cdot)u_0$. Sometimes, we prefer such estimate with the
$L^2$-norm replaced by the $L^1$-norm. The later is called the
$L^1$-observability inequality on time intervals. In Section 2, we
provide a telescoping series method, by which one can derive the
$L^1$-observability inequality on time intervals from the
$L^2$-observability inequality on time intervals.

$(5)$ Observability inequalities from time intervals for linear
parabolic equations, which grows like (\ref{wang1.3}), have been
studied in many publications (cf.,  e.g., \cite{BARBU}, \cite{DZZ},
\cite{Fernandez-CaraZuazua1}, \cite{FursikovOImanuvilov}, \cite{G.
LebeauL. Robbiano}, \cite{M1}, \cite{RR} and the references
therein). Recently, the observability inequality from measurable
sets of positive measure for the heat equation has been established
in \cite{AE}, \cite{AEWZ} and \cite{can1} (with the help of
a propagation estimate of smallness for analytic functions). For some
general parabolic equations (or systems) with time-independent and
analytic coefficients, we refer the reader to \cite{EMZ}.

The rest of this paper is organized as follows. Section 2 is devoted
to the proofs of Theorems~\ref{main} and \ref{bing1}. Section 3 presents some applications of Theorems~\ref{main} and \ref{bing1}
to time optimal control problems.

\bigskip

\noindent\textbf{Notation}.
For each measurable set $E\subset\mathbb R^n$, $\chi_E$ and $|E|$ stand for the  characteristic function
and the Lebesgue measure of the set, respectively.
For a smooth function $g:\mathbb R\rightarrow \mathbb R$, we write $g^{(\beta)}$, $\beta\in\mathbb N$,
 for the $\beta$-th order derivative. Sometimes we also write $e^{tA}$ for the semigroup  generated by  $A$, instead of $\{S(t);\; t\geq 0\}$. Write $\mathbb{R}^+$ for the interval $[0, \infty)$. Denote by $A^*$ and $B^*$ the adjoint operators of $A$ and $B$ respectively.
 Write $D(A)$ and $D(A^*)$ for the domains of $A$ and $A^*$ respectively.

\bigskip
\section{Proofs of Theorems~\ref{main} and \ref{bing1}}
In this section, we first prove Theorem~\ref{main} and
Theorem~\ref{bing1} respectively, and then introduce a telescoping
series method, by which one can derive the $L^1$-observability
inequality on time intervals from the $L^2$-observability inequality
on time intervals.

\subsection{The proof of Theorem~\ref{main} }
The proof of Theorem~\ref{main} is based on several  lemmas. The
first one concerns  with an analyticity property of the function:
\begin{equation}\label{def}
g(t; u_0)\triangleq\|BS(t)u_0\|_U^2, \;\;t>0,
\end{equation}
where $\{S(t);t\geq0\}$ is an analytic semigroup with the generator $A$, $u_0\in D(A)$ and
 $B\in\mathcal L(D(A),U)$.

\begin{Lemma}\label{a-prop1}
For each $u_0\in D(A)$, the function $g(\cdot)\triangleq g(\cdot;u_0)$ is analytic in $(0,+\infty)$. Furthermore, there are constants $K\geq1$ and $\rho\in(0,1)$ independent of $u_0$
such that
\begin{equation*}
\big|g^{(\beta)}(t)\big|\leq K\frac{(t-s)^{-2}\beta!}{\big(\rho
(t-s)\big)^\beta}\|u(s)\|_X^2 \;\;\mbox{for all}\;\; \beta\in\mathbb
N,
\end{equation*}
when $0<t-s\leq1$, where $u(\cdot)\triangleq S(\cdot)u_0$.
\end{Lemma}
\begin{proof}
By the translation, it suffices to prove the desired estimate for
the case that $s=0$ and $0<t\leq1$. We first assume that
$\{S(t);t\geq0\}$ is an uniformly bounded analytic semigroup with
\begin{equation*}
\|S(t)\|_{\mathcal{L}(X,X)}\leq M\;\;\text{for all}\;\; t>0,
\end{equation*}
 for some positive
constant $M$. By \eqref{def} and the binomial formula, we have
\begin{equation*}\label{a-1}
g^{(\beta)}(t)=\sum_{\beta_1+\beta_2=\beta}
\frac{\beta!}{\beta_1!\beta_2!}\lg
Bu^{(\beta_1)}(t),Bu^{(\beta_2)}(t)\rg_U\;\;\mbox{for all}\;\;
\beta\in\mathbb N\;\;\mbox{and}\;\;\;t\in (0,1].
\end{equation*}
It follows from the Cauchy-Schwartz inequality that for any $t\in
(0,1]$,
\begin{equation}\label{song1}
\begin{split}
|g^{(\beta)}(t)|&\leq \sum_{\beta_1+\beta_2=\beta}\frac{\beta!}{\beta_1!\beta_2!}
\|Bu^{(\beta_1)}(t)\|_U\|Bu^{(\beta_2)}(t)\|_U\\
&\leq \sum_{\beta_1+\beta_2=\beta}\frac{\beta!}{\beta_1!\beta_2!}
\|B\|^2_{\mathcal{L}(D(A),U)}\|u^{(\beta_1)}(t)\|_{D(A)}
\|u^{(\beta_2)}(t)\|_{D(A)}\\
&\leq \|B\|^2_{\mathcal{L}(D(A),U)}\!\!\!\!
\sum_{\beta_1+\beta_2=\beta}\frac{\beta!}{\beta_1!\beta_2!}
\Big[\|Au^{(\beta_1)}(t)\|_X+\|u^{(\beta_1)}(t)\|_X\Big]\\
&\;\;\;\;\;\;\;\;\;\;\;\;\;\;\;\;\;\;\;\;\;\;\;\;\;\;\;\;\;\;\;
\times\Big[\|Au^{(\beta_2)}(t)\|_X+\|u^{(\beta_2)}(t)\|_X\Big].
\end{split}
\end{equation}
Meanwhile, since
\begin{equation*}
\|AS(t)\|_{\mathcal L(X,X)}=\|S'(t)\|_{\mathcal L(X,X)}\leq \frac{C}{t},\;\;t>0,
\end{equation*}
for some  constant $C>0$  (cf., e.g., \cite[Chapter 2, Theorem
5.2]{Pa}), and because
\begin{equation*}
S^{(m)}(t)=\Big(AS\Big(\frac{t}{m}\Big)\Big)^m=
\Big(S'\Big(\frac{t}{m}\Big)\Big)^m,\; t>0,  m\in\mathbb N,
\end{equation*}
there is a constant $\rho\in(0,1)$ independent of $u_0$ such that
\begin{equation*}
\|u^{(m)}(t)\|_X\leq \Big(\frac{C}{(t/m)}\Big)^m\|u_0\|_X\leq
\frac{m!}{(\rho t)^m}\|u_0\|_X\;\;\mbox{for all}\;\; t\in
(0,1]\;\;\mbox{and}\;\; m\in\mathbb N.
\end{equation*}
In the last inequality above, we used the Stirling formula:
$m^m\lesssim e^mm!$, $m\in\mathbb N$. Consequently,
\begin{equation*}
\|Au^{(m)}(t)\|_X=\|u^{(m+1)}(t)\|_X\leq \frac{(m+1)!}{(\rho
t)^{m+1}}\|u_0\|_X\;\;\mbox{for all}\;\; t\in
(0,1]\;\;\mbox{and}\;\; m\in\mathbb N.
\end{equation*}
Along with the above two estimates, \eqref{song1} leads to
\begin{equation*}
\begin{split}
|g^{(\beta)}(t)|&\leq 4\|B\|^2_{\mathcal{L}(D(A),U)}
\sum_{\beta_1+\beta_2=\beta}\frac{\beta!}{\beta_1!\beta_2!}
\frac{(\beta_1+1)!}{(\rho t)^{\beta_1+1}}
\frac{(\beta_2+1)!}{(\rho t)^{\beta_2+1}}\|u_0\|_X^2\\
&\leq 4\|B\|^2_{\mathcal{L}(D(A),U)}\beta!(\rho t)^{-\beta-2}\|u_0\|_X^2
\sum_{\beta_1+\beta_2=\beta}(\beta_1+1)(\beta_2+1)\\
&\leq 4\|B\|^2_{\mathcal{L}(D(A),U)}\beta!\big(\rho
t/8\big)^{-\beta-2}\|u_0\|_X^2\;\;\mbox{for all}\;\; t\in (0,1].
\end{split}
\end{equation*}
Thus,
\begin{equation*}
|g^{(\beta)}(t)|\leq N\beta!\big(\rho t\big)^{-\beta}\;\;
\text{for all}\;\; \beta\in\mathbb N,\;t>0,\;\;\;\;\text{with}
\;\;N=4\|B\|^2_{\mathcal{L}(D(A),U)}(\rho t)^{-2}\|u_0\|_X^2,
\end{equation*}
for some new constant $\rho\in (0,1)$ independent of $u_0$. This implies the desired
estimate for  the case where the analytic semigroup
$\{S(t);t\geq0\}$ is uniformly bounded.

Next, we remove the assumption of the uniform boundedness  from the
analytic semigroup $\{S(t);t\geq0\}$. Since
$$
\|S(t)\|_{\mathcal{L}(X,X)}\leq Me^{\alpha t},\; t>0,
$$
for some constants $M>0$ and $\alpha>0$, the semigroup
$\{\widetilde{S}(t); t\geq 0\}$ with  $\widetilde{S}(t)\triangleq
e^{-\alpha t}S(t)$ for $t\geq0$, is  uniformly bounded and analytic.
Given $u_0\in D(A)$,  define
\begin{equation*}
\tilde{g}(t)\triangleq \|B\widetilde{S}(t)u_0\|_U^2, \;\;t>0.
\end{equation*}
It is clear that
$$
g(t)=e^{2\alpha t}\tilde{g}(t),\; t>0,
$$
where $g$ is the function given by (\ref{def}) corresponding to the
same $u_0$ as above. We have already verified that there is a
$\tilde{\rho}\in (0,1)$ independent of $u_0$ such that
\begin{equation}\label{newe1}
\big|\tilde{g}^{(\beta)}(t)\big|\leq N\beta!(\tilde{\rho}
t)^{-\beta}, \;\text{with}\;
N=4\|B\|^2_{\mathcal{L}(D(A),U)}(\tilde{\rho}
t)^{-2}\|u_0\|_X^2,\;\mbox{for all}\; \beta\in\mathbb N, \;t>0.
\end{equation}
Notice that
$$
g^{(\beta)}(t)=\sum_{\beta_1+\beta_2=\beta}\frac{\beta!}{\beta_1!\beta_2!}
(2\alpha)^{\beta_1}e^{2\alpha t}\tilde{g}^{(\beta_2)}(t),\; t>0.
$$
This, along with  \eqref{newe1}, implies the desired inequality for
the case when $s=0$ and $0<t\leq1$, and completes the proof.
\end{proof}

Next, we recall the following lemma, which is a propagation of smallness estimate from measurable sets for analytic
 functions in $\mathbb R$ (cf., e.g., \cite{V}, \cite[Lemma 2]{AE} or \cite[Lemma 13]{AEWZ}).
\begin{Lemma}\label{a-1d}
Let $f:[a,a+s]\rightarrow \mathbb{R}$, where $a\in\mathbb{R}$ and $s>0$, be an analytic function satisfying
\begin{equation*}
\big|f^{(\beta)}(x)\big|\leq M\beta!(s\rho)^{-\beta}\;\;\text{for all}\;\; x\in[a,a+s]\;\;\text{and}\;\; \beta\in\mathbb{N},
\end{equation*}
with some constants $M>0$ and $\rho\in(0,1]$. Assume that $\hat
E\subset[a,a+s]$ is a  subset of positive  measure. Then there are
two constants $C=C(\rho,|\hat E|/s)\geq 1$ and $\vartheta
=\vartheta(\rho,|\hat E|/s)$ with $\vartheta\in(0,1)$ such that
\begin{equation*}
\|f\|_{L^\infty(a,a+s)}\leq CM^{1-\vartheta}\Big(\frac{1}{|\hat
E|}\int_{\hat E}|f(x)|\,d x\Big)^{\vartheta}.
\end{equation*}
\end{Lemma}
\bigskip

When $(A,B)$ verifies the  observability inequality \eqref{wang1.3},
we can make use of   Lemma~\ref{a-prop1} and Lemma~\ref{a-1d} to
prove the interpolation inequality presented in the following lemma.

\begin{Lemma}\label{a-interpolation}
Suppose that the conditions in Theorem~\ref{main} hold. Let $0\leq
t_1<t_2$ with $0<t_2-t_1\leq1$. Assume that $E\subset[t_1,t_2]$ is a
subset of positive measure and verifies
$|E\cap(t_1,t_2)|\geq\eta(t_2-t_1)$ with $\eta\in(0,1)$. Then there
are two positive constants $C=C(d,k,\rho,\eta,\|B\|_{\mathcal
L(D(A),U)})$ (where $d, k$ are given by (\ref{wang1.3}) and $\rho$
is given by Lemma~\ref{a-prop1}) and
$\theta=\theta(\rho,\eta)\in(0,1)$ such that for any $u_0\in D(A)$,
the corresponding solution  $u$ to Equation~\eqref{a-eq1} satisfies
\begin{equation}\label{407c1}
\|u(t_2)\|_X\leq
\Big(Ce^{\frac{C}{(t_2-t_1)^k}}\int_{t_1}^{t_2}\chi_{E}(t)
\|Bu(t)\|_U\,dt\Big)^\theta\|u(t_1)\|_X^{1-\theta}.
\end{equation}
\end{Lemma}
\begin{proof}
Set
$$\tau=t_1+\frac{\eta}{10}(t_2-t_1)\;\;\text{and}\;\;\hat E=E\cap[\tau,t_2].$$
Clearly,
\begin{equation}\label{biao1}
|\hat E|\geq\frac{ \eta(t_2-t_1)}{2}.
\end{equation}
By Lemma~\ref{a-prop1}, we get that for any $t\in[\tau,t_2]$,
\begin{equation*}
\begin{split}
|g^{(\beta)}(t)|&\leq K\frac{(t-t_1)^{-2}\beta!}{(\rho (t-t_1))^\beta}\|u(t_1)\|_X^2
\leq K\frac{(\tau-t_1)^{-2}\beta!}{(\rho(\tau-t_1))^\beta}\|u(t_1)\|_X^2\\
&\leq K\frac{\big(\eta (t_2-t_1)/10\big)^{-2}\beta!}{\big(\rho\eta(t_2-\tau)/10\big)
^\beta}\|u(t_1)\|_X^2\\
&\triangleq M\beta!( s\rho_1)^{-\beta}\;\;\text{for all}\;\; \beta\in\mathbb N,\\
\end{split}
\end{equation*}
with
$$M=100\eta^{-2}(t_2-t_1)^{-2}K\|u(t_1)\|_X^2,\;\;\rho_1=\frac{\rho\eta}{10}\;\;\;\;
\text{and}\;\;\;\;s=t_2-\tau.$$ According to Lemma~\ref{a-1d}, there
are positive constants $C=C(K,\rho,\eta)$ and
$\vartheta=\vartheta(\rho,\eta)\in(0,1)$ such that
\begin{equation*}
\|g\|_{L^\infty(\tau,t_2)}\leq
(t_2-t_1)^{-2}C\|u(t_1)\|_X^{2(1-\vartheta)}\Big(\frac{1} {|\hat
E|}\int_{\hat E}|g(s)|\,d s\Big)^{\vartheta},
\end{equation*}
which is equivalent to
\begin{equation}\label{a-2}
\|Bu(t)\|_U^2\leq
(t_2-t_1)^{-2}C\|u(t_1)\|_X^{2(1-\vartheta)}\Big(\frac{1}{|\hat E|}
\int_{\hat E}\|Bu(s)\|_U^2\,d s\Big)^{\vartheta}\;\;\mbox{for
all}\;\; t\in[\tau,t_2].
\end{equation}
By the translation and the   observability inequality
\eqref{wang1.3}, we have
\begin{equation*}
\|u(t_2)\|_X^2\leq
e^{\frac{d}{(t_2-\tau)^k}}\int_{\tau}^{t_2}\|Bu(t)\|_U^2\,dt\leq
e^{\frac{C(d,k,\eta)}{(t_2-t_1)^k}}\int_{\tau}^{t_2}\|Bu(t)\|_U^2\,dt.
\end{equation*}
This, along with \eqref{biao1} and  \eqref{a-2},   leads to
\begin{equation}\label{biao2}
\begin{split}
\|u(t_2)\|_X^2
&\leq e^{\frac{C(\eta,d,k)}{(t_2-t_1)^k}}(t_2-\tau)(t_2-t_1)^{-2}C
\|u(t_1)\|_X^{2(1-\vartheta)}
\Big(\frac{1}{|\hat E|}\int_{\hat E}\|Bu(s)\|_U^2\,ds\Big)^{\vartheta}\\
&\leq
e^{\frac{C(\eta,d,k)}{(t_2-t_1)^k}}(t_2-t_1)^{-2}C\|u(t_1)\|_X^{2(1-\vartheta)}
\max_{t\in[\tau,t_2]}\|Bu(t)\|_U^{\vartheta} \Big(\int_{\hat
E}\|Bu(s)\|_U\,ds\Big)^{\vartheta}.
\end{split}
\end{equation}
By the properties of analytic semigroups, we see that for any
$t\in[\tau,t_2]$,
\begin{equation*}
\|Bu(t)\|_U\leq \|B\|_{\mathcal L(D(A),U)}\big(\|Au(t)\|_X+\|u(t)\|_X\big)
\leq C(t_2-t_1)^{-1}\|u(t_1)\|_X,
\end{equation*}
with some constant $C=C(\|B\|_{\mathcal L(D(A),U)})>0$.
This, together with \eqref{biao2}, indicates that
\begin{equation*}
\begin{split}
\|u(t_2)\|_X^2
&\leq e^{\frac{C(\eta,d,k)}{(t_2-t_1)^k}}(t_2-t_1)^{-2}C\|u(t_1)\|_X^{2(1-\vartheta)}
(t_2-t_1)^{-1}C\|u(t_1)\|_X^{\vartheta}
\Big(\int_{\hat E}\|Bu(s)\|_U\,ds\Big)^{\vartheta}\\
&\leq
(t_2-t_1)^{-3}e^{\frac{C(\eta,d,k)}{(t_2-t_1)^k}}C\|u(t_1)\|_X^{2-\vartheta}
\Big(\int_{\hat E}\|Bu(s)\|_U\,ds\Big)^{\vartheta}.
\end{split}
\end{equation*}
This, along with the estimate $(t_2-t_1)^{-3}\leq
e^{3/[k(t_2-t_1)^k]}$, leads to  \eqref{407c1}, and completes the
proof.
\end{proof}

We end this subsection with presenting the proof of
Theorem~\ref{main}. The proof is based on
Lemma~\ref{a-interpolation} and the telescoping series method
(provided in \cite{AEWZ}), which is a modified version of that in
\cite{PW2}.

\noindent
\begin{proof}[\text{The proof of Theorem~\ref{main} }]
Let $\ell\in(0,T)$ be a Lebesgue density point of $E$. Then for each
constant $q\in(0,1)$ which is to be fixed later, there exists a
monotone  decreasing sequence $\{\ell_m\}_{m\geq1}\subset(0,T)$,
with $0<\ell_1-\ell_{2}\leq1$, such that (cf. \cite[Proposition
2.1]{PW2})
\begin{equation}\label{jia4}
\ell_{m+1}-\ell_{m+2}=q(\ell_m-\ell_{m+1}),\;\;
|E\cap(\ell_{m+1},\ell_m)|\geq
\frac{\ell_{m}-\ell_{m+1}}{3}\;\;\mbox{for all}\;\; m\geq1,
\end{equation}
and such that
\begin{equation}\label{16-11}
\lim_{m\rightarrow +\infty}\ell_m=\ell.
\end{equation}

Given $u_0\in D(A)$, write  $u(\cdot)=S(\cdot)u_0$. According to
Lemma~\ref{a-interpolation},  there are constants
$C=C(d,k,\rho,\|B\|_{\mathcal L(D(A),U)})\geq1$ and $\theta=\theta(\rho)\in(0,1)$
such that  when $m\geq 1$,
\begin{equation*}
\|u(\ell_m)\|_X
\leq \Big(Ce^{\frac{C}{(\ell_m-\ell_{m+1})^k}}
\int_{\ell_{m+1}}^{\ell_m}\chi_{E}\|Bu(t)\|_U\,d t\Big)^{\theta}
\|u(\ell_{m+1})\|_X^{1-\theta}.
\end{equation*}
This, together with   Young's inequality:
\begin{equation*}
ab\leq \varepsilon a^p+\varepsilon^{-\frac{r}{p}}b^r,\;\;\text{when}\;\;a>0, b>0, \varepsilon>0,
\end{equation*}
with
\begin{equation*}
\frac{1}{p} +\frac{1}{r}=1, \ p>1, \ r>1,
\end{equation*}
indicates that when $m\geq1$,
\begin{equation*}
\|u(\ell_m)\|_X\leq \varepsilon\|u(\ell_{m+1})\|_X
+\varepsilon^{-\frac{1-\theta}{\theta}}
Ce^{\frac{C}{(\ell_m-\ell_{m+1})^k}}
\int_{\ell_{m+1}}^{\ell_m}\chi_{E}\|Bu(t)\|_U\,d t\;\;\mbox{for
all}\;\; \varepsilon>0,
\end{equation*}
which is equivalent to
\begin{equation}\label{16-12}
\begin{split}
\varepsilon^{1-\theta}e^{-\frac{C}{(\ell_m-\ell_{m+1})^k}}
\|u(\ell_{m})\|_X
&-\varepsilon e^{-\frac{C}{(\ell_m-\ell_{m+1})^k}}
\|u(\ell_{m+1})\|_X\\
&\leq C \int_{\ell_{m+1}}^{\ell_m}\chi_{E}\|Bu(t)\|_U\,d
t\;\;\mbox{for all}\;\; \varepsilon>0.
\end{split}
\end{equation}
By letting $\varepsilon=e^{-1/[(\ell_m-\ell_{m+1})^k]}$ in \eqref{16-12}, we have
\begin{equation}\label{16-13}
\begin{split}
e^{-\frac{C+1-\theta}{(\ell_m-\ell_{m+1})^k}}\|u(\ell_m)
\|_X&-e^{-\frac{C+1}{(\ell_m-\ell_{m+1})^k}}
\|u(\ell_{m+1})\|_X\\
&\leq C\int_{\ell_{m+1}}^{\ell_m}\chi_{E}\|Bu(t)\|_U\,d
t\;\;\mbox{for all}\;\; m\geq 1.
\end{split}
\end{equation}
We now take
\begin{equation*}
q=\Big(\frac{C+1-\theta}{C+1}\Big)^{\frac{1}{k}}\in (0,1),
\;\;\text{where $C$ and $\theta$ are given in}\;\;\eqref{16-13}.
\end{equation*}
It follows from \eqref{16-13} and the first formula of \eqref{jia4} that
\begin{equation*}
\begin{split}
e^{-\frac{C+1-\theta}{(\ell_m-\ell_{m+1})^k}}\|u(\ell_m)
\|_X&-e^{-\frac{C+1-\theta}{(\ell_{m+1}-\ell_{m+2})^k}}
\|u(\ell_{m+1})\|_X\\
&\leq
C\int_{\ell_{m+1}}^{\ell_m}\chi_{E}\|Bu(t)\|_U\,d t\;\;\text{for all}\;\; m\geq1.
\end{split}
\end{equation*}
Summing the above inequality from $m=1$ to $+\infty$, and noticing the convergence \eqref{16-11},
as well as
\begin{equation*}
\sup_{t\in(0,T)}\|u(t)\|_X<+\infty,
\end{equation*}
we see that
\begin{equation*}
\|u(\ell_{1})\|_X\leq
Ce^{\frac{C+1-\vartheta}{(\ell_1-\ell_2)^k}}
\int_\ell^{\ell_1}\chi_{E}\|Bu(t)\|_U\,d t.
\end{equation*}
Because $\|u(T)\|_X\leq C\|u(\ell_{1})\|_X$, the above leads to
\eqref{403c8}. This completes the proof.
\end{proof}

\bigskip
\subsection{The proof of Theorem~\ref{bing1}}
The main idea of the proof is borrowed from \cite[Theorem 6]{AEWZ}.

\begin{proof}[\text{The proof of Theorem~\ref{bing1}}]
We begin with proving  the interpolation inequality \eqref{hong1}.
For each $\lambda\geq\lambda_1$, we define
\begin{equation*}
\mathbb E_\lambda=\bigcup_{\lambda_k\leq \lambda}\mathbb E_{\lambda_k},
\end{equation*}
which is a subspace of $X$. Denote by $\mathcal E_{\lambda}$ the
orthogonal projection operator from $X$ to  $\mathbb E_{\lambda}$.
Given  $u_0\in X$,  write $\mathcal E_{\lambda}^\bot
u_0=u_0-\mathcal E_{\lambda} u_0 $. Because
\begin{equation}\label{hong3}
\|S(t)u_0\|_X\leq \|S(t)\mathcal E_{\lambda} u_0\|_X+\|S(t)\mathcal E^\bot_{\lambda} u_0\|_X,
\end{equation}
we conclude from the properties $(i)$ and $(iii)$ of Hypothesis $(H)$ that
\begin{equation*}
\begin{split}
\|S(t)\mathcal E_{\lambda} u_0 \|_X&\leq Ne^{N{\lambda}^\gamma}\|BS(t)\mathcal E_{\lambda} u_0\|_U\\
&\leq Ne^{N{\lambda}^\gamma}\big(\|BS(t)u_0\|_U+\|BS(t)\mathcal E_{\lambda}^\bot u_0\|_U\big)\\
&\leq  Ne^{N{\lambda}^\gamma}\big(\|BS(t)u_0\|_U+\|B\|_{\mathcal L(X,U)}\|S(t)\mathcal E_{\lambda}^\bot u_0\|_X\big).
\end{split}
\end{equation*}
This, together with \eqref{hong3}, implies that
\begin{equation}\label{hong4}
\|S(t)u_0\|_X\leq Ne^{N{\lambda}^\gamma}\big(1+\|B\|_{\mathcal L(X,U)}\big)
\big(\|BS(t)u_0\|_U+\|S(t)\mathcal E_{\lambda}^\bot u_0\|_X\big).
\end{equation}
By the property $(ii)$ of  Hypothesis $(H)$, we have
$$\|S(t)\mathcal E_{\lambda}^\bot u_0\|_X\leq e^{-\mu{\lambda} t}\|\mathcal E^\bot_{\lambda} u_0\|_X\leq e^{-\mu{\lambda} t}\|u_0\|_X.$$
Along with  \eqref{hong4}, this yields that for any
$\lambda\geq\lambda_1$,
\begin{equation*}
\|S(t)u_0\|_X\leq N\big(1+\|B\|_{\mathcal L(X,U)}\big)
\Big[\exp\Big(N{\lambda}^\gamma-\frac{\mu{\lambda}
t}{2}\Big)\Big]\big(e^{\mu{\lambda}
t/2}\|BS(t)u_0\|_U+e^{-\mu{\lambda} t/2}\|u_0\|_X\big).
\end{equation*}
Because
$$\max_{\lambda>0}\Big\{N{\lambda}^\gamma-\frac{\mu{\lambda} t}{2}\Big\}\leq N
\Big(\frac{2\gamma N}{\mu
t}\Big)^{\frac{\gamma}{1-\gamma}},\;\;\text{when}\;\;\gamma\in(0,1),$$
 there is a constant $K=K(N,\mu,\gamma,
\|B\|_{\mathcal L(X,U)})$ such that
$$\|S(t)u_0\|_X\leq Ke^{Kt^{-\frac{\gamma}{1-\gamma}}}\Big(e^{\mu\lambda t/2}\|BS(t)u_0\|_U
+e^{-\mu\lambda t/2}\|u_0\|_X\Big)\;\;\mbox{for all}\;\;
\lambda\geq\lambda_1,
$$
which is equivalent to
\begin{equation}\label{biao3}
\|S(t)u_0\|_X\leq
Ke^{Kt^{-\frac{\gamma}{1-\gamma}}}\Big(\varepsilon^{-1}
\|BS(t)u_0\|_U+\varepsilon\|u_0\|_X\Big)\;\;\mbox{for all}\;\;
\varepsilon\in(0,e^{-\mu\lambda_1t/2}].
\end{equation}
Since
$$
\|S(t)u_0\|_X\leq M\|u_0\|_X,\;\;\mbox{when}\;\; t\in(0,1],\;\;\mbox{for some}\;\;M>0,
$$
  it holds that for each
$t\in(0,1]$,
\begin{equation*}
\|S(t)u_0\|_X\leq Me^{\mu\lambda_1t/2}\varepsilon
\|u_0\|_X\;\;\mbox{for all}\;\; \varepsilon\geq
e^{-\mu\lambda_1t/2}.
\end{equation*}
This, combined with \eqref{biao3}, leads to
$$
\|S(t)u_0\|_X\leq
Me^{\mu\lambda_1}Ke^{Kt^{-\frac{\gamma}{1-\gamma}}}\Big(\varepsilon^{-1}
\|BS(t)u_0\|_U+\varepsilon\|u_0\|_X\Big)\;\;\mbox{for all}\;\;
\varepsilon\in(0,+\infty).
$$
Minimizing the above inequality with respect to $\varepsilon$ gives  the
desired estimate \eqref{hong1}.

We next show the  observability inequality \eqref{hong2} through
utilizing a telescoping series method. Let $\ell\in(0,T)$ be a
Lebesgue point of $E$. For each constant $q\in(0,1)$ which will be precised later,  there exists a monotone decreasing sequence
$\{\ell_m\}_{m\geq1}$ satisfying \eqref{jia4}, \eqref{16-11} and
$0<\ell_1-\ell_{2}\leq1$ (cf. \cite[Proposition 2.1]{PW2}).  Let us
set
$$\tau_m=\ell_{m+1}+\frac{\ell_m-\ell_{m+1}}{6}\;\;\text{for all}\;\; m\geq1.$$
By the inequality \eqref{hong1}, we deduce that for any $t\in[\tau_m,\ell_m]$ and any $u_0\in X$,
\begin{equation}\label{yuanyuan2.24}
\begin{split}
 \|S(t)u_0\|_X&\leq \Big(C\exp\big(C(t-\ell_{m+1})^{-\frac{\gamma}{1-\gamma}}\big)
 \|BS(t)u_0\|_{U}
\Big)^{\frac 12}\|S(\ell_{m+1})u_0\|_X^{\frac 12}\\
&\leq  \Big(N\exp\big(N(\ell_m-\ell_{m+1})^{-\frac{\gamma}
{1-\gamma}}\big)\|BS(t)u_0\|_{U}
\Big)^{\frac 12}\|S(\ell_{m+1})u_0\|_X^{\frac 12},
\end{split}
\end{equation}
with some constant $N\geq1$.
Because
$$\|S(\ell_m)u_0\|_X\leq M\|S(t)u_0\|_X\;\;\mbox{for some}\; M>0 \;\;\mbox{and for all}\;\;  t\in[\tau_m,\ell_m],
 $$
the  estimate (\ref{yuanyuan2.24}) implies that
$$
\|S(\ell_m)u_0\|_X\leq
\Big(N\exp\big(N(\ell_m-\ell_{m+1})^{-\frac{\gamma}
{1-\gamma}}\big)\|BS(t)u_0\|_{U} \Big)^{\frac
12}\|S(\ell_{m+1})u_0\|_X^{\frac 12}\;\;\mbox{for all}\;\;
t\in[\tau_m,\ell_m].
$$
Then by Young's inequality, we have
$$
\|S(\ell_m)u_0\|_X\leq \varepsilon
\|S(\ell_{m+1})u_0\|_X+\varepsilon^{-1}
N\exp\Big(\frac{N}{(\ell_m-\ell_{m+1})^{\frac{\gamma}{1-\gamma}}}\Big)
\|BS(t)u_0\|_U\;\;\mbox{for all}\;\; t\in[\tau_m,\ell_m].
$$
Integrating the above inequality over $[\tau_m,\ell_m]\cap E$ and
noting that
$$
|(\tau_m,\ell_m)\cap E|\geq (\ell_m-\ell_{m+1})/6,
$$
we obtain that for any $\varepsilon>0$,
$$
\|S(\ell_m)u_0\|_X\leq \varepsilon \|S(\ell_{m+1})u_0\|_X+\varepsilon^{-1}
N\exp\Big(\frac{N}{(\ell_m-\ell_{m+1})^{\frac{\gamma}{1-\gamma}}}\Big)
\int_{\ell_{m+1}}^{\ell_m}\chi_E(t)\|BS(t)u_0\|_U\,dt.
$$
By taking
$$\varepsilon=\exp\Big(-\frac{1}
{2(\ell_m-\ell_{m+1})^{\frac{\gamma}{1-\gamma}}}\Big)$$
in the above inequality,
we see that
\begin{equation}\label{jia5}
\begin{split}
\exp\Big(-\frac{N+\frac{1}{2}}
{(\ell_m-\ell_{m+1})^{\frac{\gamma}{1-\gamma}}}\Big)
\|S(\ell_m)u_0\|_X&-
\exp\Big(-\frac{N+1}
{(\ell_m-\ell_{m+1})^{\frac{\gamma}{1-\gamma}}}\Big)\|S(\ell_{m+1})u_0\|_X\\
&\leq N\int_{\ell_{m+1}}^{\ell_m}\chi_E(t)\|BS(t)u_0\|_U\,dt.
\end{split}
\end{equation}
We now take
$$q=\Big(\frac{N+\frac{1}{2}}{N+1}\Big)^{\frac{1-\gamma}{\gamma}}\in (0,1).$$
It follows from \eqref{jia5} that
\begin{equation*}\label{jia6}
\begin{split}
\exp\Big(-\frac{N+\frac{1}{2}}
{(\ell_m-\ell_{m+1})^{\frac{\gamma}{1-\gamma}}}\Big)
\|S(\ell_m)u_0\|_X&-
\exp\Big(-\frac{N+\frac{1}{2}}
{\big[q(\ell_m-\ell_{m+1})\big]^{\frac{\gamma}{1-\gamma}}}\Big)
\|S(\ell_{m+1})u_0\|_X\\
&\leq N\int_{\ell_{m+1}}^{\ell_m}\chi_E(t)\|BS(t)u_0\|_U\,dt.
\end{split}
\end{equation*}
Summing the above inequality with respect to $m$ from $1$ to
$+\infty$,  using \eqref{jia4} and \eqref{16-11}, we deduce the
desired estimate \eqref{hong2} immediately. This completes the
proof.
\end{proof}

\subsection{A telescoping series method}
In this subsection, we
 introduce a telescoping series method, by which one can derive the
$L^1$-observability inequality from time intervals through the
$L^2$-observability inequality from time intervals  for the equation
(\ref{a-eq1}). The main result of this subsection is as follows.
\begin{Proposition}\label{refined}
Let  $A:D(A)\subset X \rightarrow X$ generate a $C_0$ semigroup
$\{S(t);t\geq0\}$ in $X$, such that
$$
\|S(t)\|_{\mathcal L(X,X)}\leq Me^{\alpha t}\;\;\mbox{for all}\;\;
t\geq 0,
$$
where  $M>0$ and $\alpha\in\mathbb R^+$ are independent of $t$. Let
$B\in\mathcal L(X,U)$. Suppose that there are two positive constants
$d$, $k$ and a nondecreasing function   $\theta(\cdot)$ from
$\mathbb R^+$ to $\mathbb R^+$ such that
\begin{equation}\label{29ass}
\|S(L)u_0\|_X\leq
\theta(L)e^{\frac{d}{L^k}}\Big(\int_0^L\|BS(t)u_0\|
_U^2\,dt\Big)^{1/2}\;\;\mbox{for all}\;\;
L>0\;\;\mbox{and}\;\;u_0\in X.
\end{equation}
 Then there exists a positive constant $N=N(d,k)$
such that
\begin{equation}\label{403c1}
\|S(T)u_0\|_X\leq F(T)e^{\frac{N}{T^k}}\int_0^T\|BS(t)u_0\|
_U\,dt\;\;\mbox{for all}\;\;T>0\;\;\mbox{and}\;\;u_0\in X,
\end{equation}
where $F(\cdot)$ is a function defined by
\begin{equation}\label{29no5}
F(T)=\theta(T)^2\|B\|_{\mathcal L(X,U)}
 Me^{\alpha T},\; T>0.
\end{equation}
\end{Proposition}

\begin{proof}
Let $T>0$ and $u_0\in X$. For each $q\in(0,1)$, we define a sequence
of real numbers $\{\ell_m\}_{m\geq0}$ by
$$\ell_m=q^mT\;\;\mbox{for all}\;\; m\geq0.$$
Clearly,
\begin{equation}\label{29no4}
\ell_{m+1}-\ell_{m+2}=q(\ell_m-\ell_{m+1})\;\;\;\;
\text{and}\;\;\lim_{m\rightarrow
+\infty}\big(\ell_m-\ell_{m+1}\big)=0.
\end{equation}
By the translation, we see from \eqref{29ass} that for any $m\geq0$,
\begin{equation}\label{29no1}
\begin{split}
\|S(\ell_m)u_0\|_X&\leq
\theta(\ell_m-\ell_{m+1})e^{\frac{d}{(\ell_m-\ell_{m+1})^k}}
\Big(\int_{\ell_{m+1}}^{\ell_m}\|BS(t)u_0\| _U^2\,dt\Big)^{1/2}.
\end{split}
\end{equation}
Since
\begin{equation*}
\max_{t\in(\ell_{m+1},\ell_m)} \|BS(t)u_0\|_X\leq \|B\|_{\mathcal
L(X,U)}
 Me^{\alpha (\ell_m-\ell_{m+1})}\|S(\ell_{m+1})u_0\|_X,
\end{equation*}
the estimate (\ref{29no1}), together with \eqref{29no5}, leads to
\begin{equation}\label{29no2}
\|S(\ell_m)u_0\|_X\leq \Big(F(\ell_m-\ell_{m+1})
e^{\frac{2d}{(\ell_m-\ell_{m+1})^k}}\int_{\ell_{m+1}}^{\ell_m}
\|BS(t)u_0\|_U\,dt\Big)^{1/2} \|S(\ell_{m+1})u_0\|_X^{1/2}.
\end{equation}
Because
 $$
 F(\ell_m-\ell_{m+1})\leq F(T)\;\;\mbox{for all}\;\;  m\geq0,
 $$
by  applying the  Young inequality to \eqref{29no2}, we see  that
\begin{equation*}
\|S(\ell_m)u_0\|_X\leq \varepsilon\|S(\ell_{m+1})u_0\|_X
+\varepsilon^{-1}
F(T)e^{\frac{2d}{(\ell_m-\ell_{m+1})^k}}\int_{\ell_{m+1}}^{\ell_m}
\|BS(t)u_0\|_U\,dt\;\;\mbox{for each}\;\; \varepsilon>0.
\end{equation*}
Multiplying the above inequality by $\varepsilon
e^{-\frac{2d}{(\ell_m-\ell_{m+1})^k}}$ and then taking
$\varepsilon=e^{-\frac{1}{(\ell_m-\ell_{m+1})^k}}$ in the resulting
inequality, we obtain that
\begin{equation}\label{29no3}
 e^{-\frac{2d+1}{(\ell_m-\ell_{m+1})^k}}\|S(\ell_m)u_0\|_X
- e^{-\frac{2d+2}{(\ell_m-\ell_{m+1})^k}}\|S(\ell_{m+1})u_0\|_X \leq
F(T)\int_{\ell_{m+1}}^{\ell_m} \|BS(t)u_0\|_U\,dt.
\end{equation}
Now, we choose
$$q=\Big(\frac{2d+1}{2d+2}\Big)^{\frac{1}{k}}.$$
It is obvious that $q\in(0,1)$. Therefore, it follows from
\eqref{29no3} and \eqref{29no4} that
\begin{equation*}
e^{-\frac{2d+1}{(\ell_m-\ell_{m+1})^k}}\|S(\ell_m)u_0\|_X -
e^{-\frac{2d+1}{(\ell_{m+1}-\ell_{m+2})^k}}\|S(\ell_{m+1})u_0\|_X
\leq F(T)\int_{\ell_{m+1}}^{\ell_m} \|BS(t)u_0\|_U\,dt.
\end{equation*}
Summing the above inequality with respect to $m$ from $0$ to
$+\infty$ (the telescoping series) and noting that
$$\lim_{m\rightarrow+\infty}
e^{-\frac{2d+1}{(\ell_{m+1}-\ell_{m+2})^k}}=0\;\;\mbox{and}\;\;
\max_{t\in[0,T]}\|S(t)u_0\|_X<+\infty,$$ we derive that
\begin{equation*}
\|S(T)u_0\|_X\leq
F(T)e^{\frac{2d+1}{[(1-q)T]^k}}\int_0^T\|BS(t)u_0\|_U\,dt.
\end{equation*}
This leads to  \eqref{403c1} and completes the proof.
\end{proof}

\begin{Remark}\label{yuanyuan2.5}
It is worth mentioning that in Proposition~\ref{refined}, the pair
$(A,B)$ does not hold conditions in either Theorem~\ref{main} or
Theorem~\ref{bing1}.
\end{Remark}

We next give two  applications of Proposition~\ref{refined}, as well
as the telescoping series method presenting in the proof of this proposition.

\bigskip

\noindent Example 2.1. Let $\Omega\subset\mathbb R^n$ be a bounded
domain with a smooth boundary $\partial\Omega$, and let $\omega$ be
a nonempty open subset of $\Omega$. Consider the following
Schr\"{o}dinger equation
\begin{equation}\label{403c3}
\begin{cases}
iu_t+\Delta u
=0\;\;&\text{in}\;\;\Omega\times\mathbb R,\\
u=0\;\;&\text{on}\;\;\partial\Omega\times\mathbb R,\\
u(\cdot,0)=u_0      \;\;&\text{in}\;\;\Omega.\\
\end{cases}
\end{equation}
Under the {\it geometric optic} condition on $\Omega$ and $\omega$,
it follows from \cite[Theorem 1.3]{M1} that there exists a positive
constant $C=C(\Omega,\omega)$ such that for any $u_0\in
L^2(\Omega)$, the corresponding solution $u$ to Equation
\eqref{403c3} verifies
\begin{equation*}
\|u(\cdot,L)\|_{L^2(\Omega)} \leq Ce^{\frac{C}{L}}
\Big(\int_0^L\int_{\omega}|u(x,t)|^2\,dxdt\Big)^{1/2}\;\;\mbox{for
all}\;\; L\in(0,1].
\end{equation*}
According to Proposition~\ref{refined} (with $X=L^2(\Omega)$,
$U=L^2(\omega)$, $A=i\Delta$ and $B=\chi_\omega I$, here $I$ is the
identity on $X$ and $\chi_\omega$ is the characteristic function of
$\omega$), it holds that for each $u_0\in L^2(\Omega)$,
\begin{equation*}
\|u(\cdot,T)\|_{L^2(\Omega)}\leq Ce^{\frac{C}{T}}
\int_0^T\|u(\cdot,t)\|_{L^2(\omega)}\,dt\;\;\mbox{for all}\;\;
T\in(0,1].
\end{equation*}
 Because of the property of isometry:
 $$
 \|u(\cdot,t)\|_{L^2(\Omega)}=\|u_0\|_{L^2(\Omega)}\;\;\mbox{for all}\;\; t>0,
 $$
 we find that for each $u_0\in L^2(\Omega)$,
$$\|u_0\|_{L^2(\Omega)}\leq
Ce^{\frac{C}{T}}
\int_0^T\|u(\cdot,t)\|_{L^2(\omega)}\,dt\;\;\mbox{for all}\;\;
T\in(0,1].
$$
With regard to the  observability for the Schr\"{o}dinger equation,
we also would like to mention \cite{Ph1} and \cite[Proposition
2.2]{LT}.

\vskip 10 pt

\noindent Example 2.2. Let  $\Omega\subset\mathbb R^n$ be a bounded
domain with a $C^2$-boundary $\partial\Omega$. Consider the
parabolic equation:
\begin{equation}\label{403c2}
\begin{cases}
u_t-\mbox{div}(\vec{a}(x)\nabla u)
=0\;\;&\text{in}\;\;\Omega\times\mathbb R^+,\\
u=0\;\;&\text{on}\;\;\partial\Omega\times\mathbb R^+,\\
u(\cdot,0)=u_0,      \;\;&\text{in}\;\;\Omega,\\
\end{cases}
\end{equation}
where $\vec{a}(\cdot)\triangleq(a_{ij}(\cdot))\in
C^1(\overline\Omega; \mathbb{R}^{n\times n})$  are such that
$a_{ij}=a_{ji}$ over $\overline\Omega$ for all $i,j$ and such that
for some $0< \mu_1<\mu_2$,
$$
\mu_1\sum\limits_{i=1}^n\xi_i^2\leq \sum\limits_{i,
j=1}^na_{ij}(x)\xi_i\xi_j\leq \mu_2\sum\limits_{i=1}^n\xi_i^2
\;\;\mbox{for all}\;\; (\xi_1,\dots, \xi_n)\in \mathbb{R}^n, \
x\in\overline\Omega.
$$
Let $\omega$ be a nonempty and open subset of $\Omega$. The
following observability inequality from time intervals has been
proved (cf. \cite[Theorem 2.1]{DZZ}): There is a constant
$C=C(\Omega,\omega,\mu_1,\mu_2)\geq1$ such that for each $u_0\in
L^2(\Omega)$, the corresponding solution $u$ to Equation
\eqref{403c2} verifies
\begin{equation*}
\|u(\cdot,L)\|_{L^2(\Omega)}\leq
Ce^{\frac{C}{L}}\Big(\int_{0}^L\int_{\omega}
|u(x,t)|^2\,dxdt\Big)^{1/2}\;\;\mbox{for all}\;\; L>0.
\end{equation*}
From this, we can apply  Proposition~\ref{refined} to get that for
each $u_0\in L^2(\Omega)$, the corresponding solution $u$ to
Equation \eqref{403c2} verifies
\begin{equation*}
\|u(\cdot,T)\|_{L^2(\Omega)}\leq Ce^{\frac{C}{T}}\int_{0}^T
\|u(\cdot,t)\|_{L^2(\omega)}dt\;\;\mbox{for all}\;\; T>0.
\end{equation*}
 Then from Nash's  inequality:
\begin{equation*}
\|g\|_{L^2(\omega)}\leq C(\Omega,\omega,\widetilde{\omega})
\|g\|_{L^1(\widetilde{\omega})}^\theta \|\nabla
g\|_{L^2(\Omega)}^{1-\theta},\;\;\text{with}\;\;\theta=\frac{2}{n+2},
\;\;\mbox{for all}\;\; g\in H_0^1(\Omega),
\end{equation*}
(where  $\widetilde{\omega}$ is a nonempty open subset satisfying
$\omega\subset\subset\widetilde\omega\subset\Omega$,) H\"{o}lder's
inequality and the standard energy estimate for solutions to
Equation \eqref{403c2}:
\begin{equation*}
\|u\|_{L^2(0,T;H_0^1(\Omega))}\leq C\|u_0\|_{L^2(\Omega)},
\end{equation*}
it follows that
\begin{equation*}
\|u(\cdot,T)\|_{L^2(\Omega)}\leq \Big(Ce^{\frac{C}{T}}\int_0^T
\int_{\widetilde\omega}|u(x,t)|\,dxdt\Big)^{\theta}\|u_0\|
_{L^2(\Omega)}^{1-\theta} \;\;\text{for
all}\;\;T>0\;\;\text{and}\;\;u_0\in L^2(\Omega).
\end{equation*}
Finally, making use of the telescoping series method provided in the
proof of Proposition~\ref{refined}, we obtain the refined
observability inequality:
\begin{equation*}
\|u(\cdot,T)\|_{L^2(\Omega)}\leq
Ce^{\frac{C}{T}}\int_{0}^T\int_{\widetilde\omega}|u(x,t)|\,dxdt\;\;\text{for
all}\;\;T>0\;\;\text{and}\;\;u_0\in L^2(\Omega).
\end{equation*}
This inequality has been  built up respectively in \cite{BARBU} and
\cite{Fernandez-CaraZuazua1} by different methods from ours.

\section{Applications of Theorems~\ref{main} and \ref{bing1} }

\subsection{Time optimal control problems in  Hilbert spaces}

We first set up a time optimal control problem for a controlled
evolution equation. Let $X$ and $U$ be two Hilbert spaces (which are
identified with their dual spaces) and $A$ generate a $C_0$
semigroup $\{ S(t); t\geq 0\}$ on $X$.   Denote by $X_{-1}$ the dual
of $D(A^*)$ with respect to the pivot space $X$. Then $\{S(t); \;
t\geq 0\}$ can be extended into a $C_0$ semigroup on $X_{-1}$ (cf.
\cite[Proposition 2.10.4]{TG}). We still use $\{S(t); \; t\geq 0\}$
to denote the extended semigroup. Let $B\in\mathcal L(U,X_{-1})$  be
an admissible control operator for $\{S(t); \; t\geq 0\}$ (cf.,
e.g., \cite[Definition 4.2.1]{TG}), i.e.,  there is a $\tau>0$ such
that  $\text{Ran}\,\Psi_\tau\subset X$, where
$$\Psi_\tau f=\int_0^\tau S(\tau-t)Bf(t)\,dt,\;\;f\in L^2(0,\tau; U).$$
 The controlled equation reads:
\begin{equation}\label{a0}
\frac{dz}{dt}=Az+Bf,\;t>0,\;\;z(0)=z_0.
\end{equation}
Here, $z_0\in X$ and $f\in L^2_{loc}(\mathbb{R}^+; U)$.
Write $z(\cdot\,; f,0,z_0)\in C(\mathbb{R}^+; X)$ for the unique solution of the equation (\ref{a0}) corresponding to $f$ and $z_0$ (cf. \cite[Theorem 2.37]{Coron} or \cite[Proposition 4.2.5]{TG}). The time optimal control problem is as
$$
(TP)^M:\;\;\;\;\;\;T(M)\triangleq\inf_{ f\in\mathcal{U}_{M}} \big\{ t>0\;:\;
z(t;f,0,z_0)=z_1\big\},
$$
where $z_1\in X$ is the target which differs from $z_0$ and
$$
\mathcal{U}_{M}=\big\{f:\mathbb R^+\rightarrow
U\;\;\text{measurable}\;:\; \|f(t)\|_U\leq M,\;\text{a.e.}\,\, t>0
\big\},\;\;\mbox{with}\;\; M>0.
$$
In this problem, $T(M)$ is called the optimal time, $f^*\in
\mathcal{U}_M$ is called an optimal control if $z(T(M); f^*,0,z_0)=z_1$.
{\it We say that  the problem $(TP)^M$ holds the bang-bang property
if any optimal control $f^*$ to this problem  verifies
$\|f^*(t)\|_U=M$ for a.e. $t\in (0, T(M))$. }

The bang-bang property is very important in  studies of time optimal
control problems. For instance, the uniqueness of the optimal
control follows immediately from this property;  some equivalence of
several different kinds of optimal control problems can be derived
with the aid of this property (cf. \cite{WX}, \cite{WZ},
\cite{can2}, \cite{ZB}). The bang-bang property for the problem
$(TP)^M$ (where $X=U$ is a Banach space and $B$ is the identity on
$X$) was first established  in \cite{Fa1964} via a very special and
smart way. It was first realized in \cite{MSE} that the bang-bang
property can be derived from the observability inequality from
measurable sets in time. When the target $z_1$ is replaced by a ball
in $X$, the bang-bang property follows from Pontryagin's maximum
principle and the unique continuation property of adjoint equations.
With respect to studies on the bang-bang property, we refer the
readers to \cite{AEWZ,EMZ,F,luqi2,LT,Micu,PW2,PWZ,gengshengwang1}
(where the target is allowed to be a single point in the state
space) and \cite{kw1,KW2,WW} (where the target is a ball in the
state space).

Our main results about the problem $(TP)^M$ are as follows.

\begin{Theorem}\label{thejingjingjingcan1}
Let $A$ generate an analytic semigroup $\{S(t);t\geq0\}$ in $X$. Let
$B\in\mathcal L(U,X_{-1})$ be an admissible control operator for
$\{S(t);t\geq0\}$. Assume that  $(A^*,B^*)$ satisfies the
observability inequality from time intervals:
\begin{equation}\label{wang3.3}
\|S(L)^*\varphi_0\|^2_{X}\leq e^{\frac{d}{L^k}}
\int_0^L\|B^*S(t)^*\varphi_0\|_{U}^2\,dt\;\;\mbox{for all}\;\;
\varphi_0\in D(A^*)\;\;\mbox{and}\;\; L\in(0,1],
\end{equation}
 where positive constants $d$ and $k$ are independent of $L$ and $\varphi_0$. Then the problem
 $(TP)^M$ holds the bang-bang property.
\end{Theorem}

\begin{Theorem}\label{speth}
Let $A$ generate a $C_0$ semigroup $\{S(t);t\geq0\}$ in $X$ and
$B\in\mathcal L(U,X)$.  Assume the pair $(A^*,B^*)$ verifies the
Hypothesis $(H)$.  Then the problem
 $(TP)^M$ holds the bang-bang property.
\end{Theorem}

The proofs of the above theorems are based on the null
controllability of the equation (\ref{a0}) from measurable sets in
time, which is equivalent to the observability inequality from
measurable sets in time for the dual equation of (\ref{a0}) (cf.,
e.g.,  \cite[Theorem 2.44]{Coron} or \cite[Theorem 11.2.1]{TG}). The
latter has been built up in Theorem~\ref{main} and
Theorem~\ref{bing1} from different cases. Though the above theorems
can be proved by the  standard way (cf. \cite{MSE},
\cite{gengshengwang1}), we provide the proof of
Theorem~\ref{thejingjingjingcan1} for the completeness of the
current paper.

\begin{proof}[{\it The proof of Theorem~\ref{thejingjingjingcan1}}]
Since $A$ generates an analytic semigroup  $\{S(t);t\geq0\}$ in $X$,
it follows from Theorem 5.2 of Chapter 2 and Lemma 10.2 of Chapter 1
in \cite{Pa} that the semigroup $\{S(t)^*; t\geq0\}$ generated by
$A^*$ is also analytic. Because $B\in\mathcal L(U,X_{-1})$ is an
admissible control operator for $\{S(t);t\geq0\}$, it follows from
Theorem 4.4.3 in \cite{TG} that $B^*\in\mathcal L(D(A^*),U)$ is an
admissible observation operator for $\{S(t)^*;t\geq0\}$. From these,
as well as (\ref{wang3.3}), we can apply Theorem~\ref{main} to get
the observability inequality from measurable sets in time for the
pair $(A^*, B^*)$ (i.e, the inequality (\ref{403c8}) with $(A,B)$
replaced by $(A^*,B^*)$):
{\it Given $T>0$ and $E\subset(0,T)$ of positive measure, there
exists a constant $C=C(T,E,k,d,\|B^*\|_{\mathcal L(D(A^*),U)})$ such
that
\begin{equation}\label{a1}
\|S(T)^*\varphi_0\|_{X}\leq C\int_{E}\|B^*S(t)^*\varphi_0\|_{U}\,dt\;\;
\text{for all}\;\; \varphi_0\in D(A^*).
\end{equation}}

Let $f^*$ be an optimal control for $(TP)^M$. We aim to show  that
$\|f^*(t)\|_U=M$ for a.e. $t\in(0,T(M))$. Seeking for  a
contradiction, we suppose that this did not hold. Then there would
exist an $\varepsilon>0$ and a subset $E\subset(0,T(M))$ of positive
measure such that
$$\|f^*(t)\|_U\leq M-\varepsilon\;\;\mbox{for each}\;\; t\in E. $$
Set $\delta_0=|E|/2$ and $\hat{E}=E\cap(\delta_0,T(M))$. Clearly,
$|\hat{E}|>0$.
Write $z^*(\cdot)\triangleq z(\cdot\,;f^*,0,z_0)$. Then $z^*(T(M))=z_1$.
By (\ref{a1}) and by the equivalence of the null
controllability and the observability inequality (cf., e.g.,
\cite[Theorem 2.44]{Coron} or \cite[Theorem 11.2.1]{TG}), we obtain
the null controllability from measurable sets for the pair $(A,B)$,
i.e., for each constant $\delta\in(0,\delta_0)$,  there is a control
$f$, with
$$
\|f\|_{L^\infty(\mathbb R^+;U)}\leq
C\|z_0-z^*(\delta)\|_X\;\;\mbox{for some}\;\;
C>0\;\;\mbox{independent of }\; \delta,
$$
such that $z(\cdot)\triangleq z(\cdot\,; f\chi_{\hat{E}},\delta,
z_0-z^*(\delta))$ verifies $z(T(M))=0$. Let
$\hat{f}=f^*+f\chi_{\hat{E}}$ and $w=z^*+z$. Then
$$\frac{dw}{dt}=Aw+B\hat{f}\;\;\mbox{over}\;\;(\delta, T(M)),\;\;\;\;w(\delta)=z_0,\;\;
w(T(M))=z_1.
$$ It is easy to verify that
$\|\hat{f}\|_{L^\infty(\delta,T(M);U)}\leq M$ when $\delta>0$ is
small enough. Finally, by setting $\tilde f(t)=\hat{f}(t+\delta)$
and $\tilde z(t)=w(t+\delta)$, $t\in(0,T(M)-\delta)$, we have
$$\frac{d\tilde z}{dt}=A\tilde z+B\tilde f\;\;\mbox{over}\;\; (0, T(M)-\delta),\;\;\;\;\tilde z(0)=z_0,\;\;
\tilde z(T(M)-\delta)=z_1.$$ This leads to a contradiction with the
optimality of $T(M)$ for $(TP)^M$, and completes the proof.
\end{proof}

\bigskip
\subsection{Examples}
This subsection presents some examples which are under the framework
of  Theorem~\ref{thejingjingjingcan1} or Theorem~\ref{speth}.

\subsubsection{Time optimal boundary control problem for the heat equation}
Let $\Omega\subset\mathbb R^n$ be a bounded domain with a smooth
boundary $\partial\Omega$. Let $\Gamma\subset \partial\Omega$ be a
nonempty open subset. For each $M>0$, we define
$$
\mathcal{U}_M=\big\{f:\mathbb R^+\rightarrow
L^2(\Gamma)\;\;\text{measurable}:\; \;\|f(t)\|_{L^2(\Gamma)} \leq M\
\text{for a.e.} \ t>0\big\}.
$$
The time optimal boundary control problem reads:
$$(TP)^M_1:\;\;\;\;\;T(M)\triangleq\inf_{f\in\mathcal{U}_M}\big\{t>0\;:\; y(t;f)=0\big\},$$
where $y(\cdot;f)$ solves the equation
\begin{equation}\label{a3}
\begin{cases}
y_t - \Delta y=0,\ &\text{in}\ \Omega\times \R^+,\\
y=f,\  &\text{on}\ \Gamma\times\R^+,\\
y=0,\ &\text{on}\ (\partial\Omega\setminus\Gamma)\times\R^+,\\
y(0)=y_0,\ &\text{in}\ \Omega,
\end{cases}
\end{equation}
where $y_0\in L^2(\Omega)\setminus\{0\}$ is arbitrarily fixed.

Let $X=H^{-1}(\Omega)$, $U=L^2(\Gamma)$,  $A=\Delta$, with $D(A)=H^1_0(\Omega)$,
$B=-\Delta D$, with  $D$  the Dirichlet map.
The space $L^2(\Gamma)$ is regarded as a subspace of $L^2(\partial\Omega)$
 by extending any element $f\in L^2(\Gamma)$ to be zero outside $\Gamma$.
 Then, from \cite[Proposition 10.7.1]{TG},
$A$  is positive, and consequently generates an analytic semigroup
$\{S(t);t\geq0\}$ in $X$; $B\in\mathcal L(U,X_{-1})$  is an
admissible control operator
 for $\{S(t);t\geq0\}$; and the equation (\ref{a3})
can be rewritten as
\begin{equation}\label{26jia1}
 \frac{dy}{dt}=Ay+Bf, \,t>0;\;\;  y(0)=y_0.
\end{equation}
 Using  \cite[Theorem
3.2]{Seidman-2008} (see also \cite[Remark 2]{AEWZ}) and then
modifying slightly the proof of \cite[Proposition 11.5.4]{TG},  we
can easily verify that the pair $(A,B)$ is null controllable in any
time interval  $(0, L)$, and the cost of the fast control is
$e^{C/L}$,
 where
$C=C(\Omega,\Gamma)>0$. By the equivalence of the null
controllability and the observability inequality (cf., e.g.,
\cite[Theorem 2.44]{Coron} or \cite[Theorem 11.2.1]{TG}),
$(A^*,B^*)$ verifies the
 observability inequality
\begin{equation*}
\|e^{LA^* }\varphi_0\|^2_{X}\leq e^{\frac{C}{L}}\int_0^L\|B^*e^{tA^* }\varphi_0\|_{U}^2\,dt\;\;\text{for all}\;\;
\varphi_0\in D(A^*)\;\;\text{and}\;\; L\in(0,1].
\end{equation*}
Then, one can utilize Theorem~\ref{thejingjingjingcan1} to derive the following result:
\begin{Corollary}
The problem $(TP)^M_1$ holds the bang-bang property.
\end{Corollary}

\bigskip

\subsubsection{The $3$-dimensional Stokes system with $2$ scalar controls}

Assume that $\Omega\subset\mathbb R^3$ is a bounded domain with a
smooth boundary $\partial\Omega$. Let $\omega\subset \Omega$ be a
nonempty open subset with its characteristic function $\chi_\omega$.
Treat $L^2(\omega)$  as a subspace of $L^2(\Omega)$ by extending
functions in $L^2(\omega)$ to be zero outside $\omega$. Consider the
controlled Stokes system
\begin{equation}\label{405c2}
\begin{cases}
y_t-\Delta y+\nabla p= f,\;\;&\text{in}\;\;\Omega\times\mathbb R^+,\\
\text{div}\,y=0,\;\;&\text{on}\;\;\Omega\times\mathbb R^+,\\
y=0,\;\;&\text{in}\;\;\partial\Omega\times\mathbb R^+,\\
y(\cdot,0)=y_0,\;\;&\text{in}\;\;\Omega,
\end{cases}
\end{equation}
where $y_0$ is arbitrarily fixed in the space:
$$
L^2_\sigma(\Omega)\triangleq\{y\in (L^2(\Omega))^3:
\;\text{div}\,y=0,\,y\cdot\nu=0\;\text{on}\;\partial\Omega\},
$$
and
$f$ is  taken from the control constraint set:
\begin{multline*}
\mathcal U_M\triangleq\Big\{f=(0,f_2,f_3)\in L^\infty(\mathbb
R^+;(L^2(\omega))^3):\; \| f(t)\|_{(L^2(\omega))^3}\leq
M \;\mbox{for a.e.}\;t>0  \Big\},
\end{multline*}
with $M>0$. The time optimal control problem reads:
\begin{equation*}
(TP)^M_2:\;\;\;\;\;T(M)\triangleq \inf_{f\in\mathcal{U}_M}\big\{t>0:\;y(t;f)=0\big\},
\end{equation*}
where $y(\cdot;f)$ is the solution to Equation~\eqref{405c2} corresponding to the control $f$.

Write $X=L^2_\sigma(\Omega)$ and $U=\{0\}\times L^2(\omega)\times
L^2(\omega)$. Define the operator  $A$ on $X$ by
\begin{equation*}
\begin{cases}
D(A)=\big(H^2(\Omega)\bigcap H_0^1(\Omega)\big)^3\bigcap L^2_\sigma(\Omega),\\
Ay=P(\Delta y)\;\;\text{for all}\;\; y\in D(A),
\end{cases}
\end{equation*}
where $P$ is the Helmholtz projection operator from $(L^2(\Omega))^3$ into $X$ (cf., e.g.,
\cite[Chapter 3]{Sohr}). Let $B\in \mathcal {L}(U,X)$ be defined by
$Bf=Pf$ for all $f\in U$ (i.e., $B$ is the composition of the Helmholtz projection operator and the imbedding of $U$ into $(L^2(\Omega))^3$).
Clearly,  $A$ is self-adjoint and
generates an analytic semigroup in $X$ (cf., e.g., \cite{giga1981});
$B$ is an admissible control operator for $\{e^{tA};t\geq0\}$ and $B^*: X\rightarrow U$
is given by
$$
B^*\varphi=(0,\chi_\omega \varphi_2,\chi_\omega \varphi_3)\;\;\mbox{
for all}\;\;\varphi=(\varphi_1,\varphi_2,\varphi_3)\in X;
$$
and the
equation~\eqref{405c2} can be rewritten as (cf., e.g., \cite[Chapter
4, Section 1.5]{Sohr})
\begin{equation*}
\frac{dy}{dt}=Ay+Bf,\;\;t>0.\ \
y(0)=y_0.
\end{equation*}
Meanwhile, it follows from \cite[Theorem 1]{CG} that there exists a
positive constant $C=C(\Omega,\omega)$ such that for each
$L\in(0,1]$,
\begin{equation*}
\sum_{j=1}^3\int_\Omega|\varphi_j(x,L)|^2\,dx
\leq e^{\frac{C}{L^9}}
\int_{0}^L\int_\omega
|\varphi_2(x,t)|^2+|\varphi_3(x,t)|^2\,dxdt\;\;\text{for all}\;\; \varphi_0\in L^2_\sigma(\Omega),
\end{equation*}
where $\varphi=(\varphi_1,\varphi_2,\varphi_3)$ solves the equation
\begin{equation*}\label{stokes}
\begin{cases}
\varphi_t-\Delta \varphi+\nabla p=0,\;\;&\text{in}\;\;\Omega\times(0,L),\\
\text{div}\,\varphi=0,\;\;&\text{in}\;\;\Omega\times(0,L),\\
\varphi=0,\;\;&\text{in}\;\;\partial\Omega\times(0,L),\\
\varphi(\cdot,0)=\varphi_0.
\end{cases}
\end{equation*}
In other words, the pair $(A^*,B^*)$ satisfies observability inequality:
\begin{equation*}
\|e^{LA^*}\varphi_0\|^2_X\leq e^{\frac{C}{L^9}}\int_0^L\|B^*e^{tA^*}\varphi_0\|_U^2\,dt
\;\;\mbox{for all}\;\; \varphi_0\in X\;\;\text{and}\;\; L\in(0,1].
\end{equation*}
Therefore, we can apply Theorem~\ref{thejingjingjingcan1} to get
\begin{Corollary}
 Problem $(TP)^M_2$  has the bang-bang property.
\end{Corollary}

\bigskip

\subsubsection{Parabolic equations associated with second order elliptic operators}
Let $\Omega\subset\mathbb R^n$ be a bounded domain with a smooth
boundary $\partial\Omega$, and $\omega$ be a nonempty open subset of
$\Omega$. Regard $L^2(\omega)$ as a subspace of $L^2(\Omega)$ by
extending functions in $L^2(\omega)$ to be zero outside $\omega$.
Consider the second order elliptic differential operator
$$ \mathbb L y=\sum_{i,j=1}^n\text{div}\big(a_{ij}(x)\nabla y\big)
+\sum_{i=1}^nb_i(x)\partial_{x_i}y+c(x)y.$$
Here,  all the coefficients belong to $C^2(\overline{\Omega})$; $a_{ij}(x)=a_{ji}(x)$,
 when $1\leq i,j\leq n$ and $ x\in\Omega$; and
 \begin{equation*}
\sum_{i,j=1}^na_{ij}(x)\xi_i\xi_j\geq \theta|\xi|^2 \;\text{for
all}\;\;x\in\Omega,\xi\in\mathbb R^n,\;\;\mbox{with}\;\; \theta>0.
\end{equation*}
The controlled parabolic equation is as
\begin{equation}\label{c-j1}
\begin{cases}
y_t-\mathbb Ly= f,\;\;&\text{in}\;\;\Omega\times\mathbb R^+,\\
y=0,\;\;&\text{on}\;\;\partial\Omega\times\mathbb R^+,\\
y(\cdot,0)=y_0,\;\;&\text{in}\;\;\Omega,
\end{cases}
\end{equation}
where  $y_0\in L^2(\Omega)\setminus\{0\}$ and $f$ is a control function taken from
$$\mathcal U_M\triangleq\big\{f\in L^\infty(\mathbb R^+;L^2(\omega)):\,
\| f(t)\|_{L^2(\omega)}\leq M,\, \mbox{a.e.} \,\ t>0
\big\},\;\;\mbox{with}\;\; M>0.
$$
 We are concerned with the time optimal control problem
$$(TP)_3^M:\;\;\;\;\quad\;T(M)\triangleq\inf_{f\in\mathcal U_M}\big\{t>0:
y(t;f)=0\big\},
$$
where $y(\cdot,f)$ is the solution to Equation~\eqref{c-j1} corresponding to the control $f$.

Let $X=L^2(\Omega)$ and $U=L^2(\omega)$. Define the operator $A$ on $X$ by setting
\begin{equation*}
\begin{cases}
D(A)=H^2(\Omega)\bigcap H_0^1(\Omega),\\
Ay= \mathbb Ly\;\;\mbox{for all}\;\; y\in D(A).
\end{cases}
\end{equation*}
Let $B\in \mathcal L(U,X)$ be defined by $Bf=f$ for all $f\in U$ (i.e., $B$ is the imbedding of $U$ into $X$).
Clearly,  $A$ generates an analytic semigroup in $L^2(\Omega)$ (cf., e.g., \cite[Chapter 7, Theorem 3.5]{Pa});
$B$ is an admissible
control operator for $\{e^{tA};t\geq0\}$ and $B^*: X\rightarrow U$ is given by $B^*\varphi=\chi_\omega\varphi$ for all $\varphi\in X$
(i.e., $B^*$ is the restriction from $X$ to $U$); and  Equation~\eqref{c-j1} can be rewritten as
\begin{equation*}
\frac{dy}{dt}=Ay+Bf,\ t>0, \ y(0)=y_0.
\end{equation*}
 Meanwhile, according to \cite[Theorem 2.1]{DZZ},  there exists a constant $C=C(\Omega,\omega)>0$ such that for each $L\in (0, 1]$,
\begin{equation*}
\|e^{LA^*}\varphi_0\|^2_X\leq
e^{\frac{C}{L}}\int_{0}^L\|B^*e^{tA^*}\varphi_0\|_U^2dt,\;\;\text{when}\;\;
\varphi_0\in X.
\end{equation*}
Hence, we have  the following consequence of
Theorem~\ref{thejingjingjingcan1}.
\begin{Corollary}
Any time optimal control $f^*$ of Problem $(TP)^M_3$ verifies the bang-bang property: $\| f^*(t)\|_{L^2(\omega)}=M$, a.e. $t\in(0,T(M))$.
\end{Corollary}


\subsubsection{Degenerate parabolic equations associated with the Grushin operator}
 Let $\gamma\in(0,1)$, $\Omega=(-1,1)\times(0,1)$ and $\omega=(a,b)\times(0,1)$, $0<a<b<1$. Treat $L^2(\omega)$  as a subspace
 of $L^2(\Omega)$ by extending functions in $L^2(\omega)$ to be zero outside $\omega$.  Consider the controlled system
\begin{equation}\label{405c9}
\begin{cases}
z_t-\partial_x^2z-|x|^{2\gamma}\partial_y^2z= f,\;\;&\text{in}\;\;
\Omega\times\mathbb R^+,\\
z=0,\;\;&\text{on}\;\;\partial\Omega\times\mathbb R^+,\\
z(x,y,0)=z_0,\;\;&\text{in}\;\;\Omega,
\end{cases}
\end{equation}
where  $z_0\in L^2(\Omega)\setminus\{0\}$ and the control function $f$ is taken from
$$\mathcal U_M\triangleq\big\{f\in L^\infty(\mathbb R^+;L^2(\omega)):\,
\| f(t)\|_{L^2(\omega)}\leq M\,\,\mbox{for a.e.}\;\; t>0
\big\},\;\;\mbox{with}\;\; M>0.
$$
 We are interested in the time optimal control problem
$$(TP)^M_4:\;\;\;\;\quad\;T(M)\triangleq\inf_{f\in\mathcal U_M}\big\{t>0:
z(t;f)=0\big\},
$$
where $z(\cdot;f)$ is the solution to Equation~\eqref{405c9} corresponding to the control $f$.

We next recall the well-posedness of Equation~\eqref{405c9} (see
\cite[Section 2.1]{BCG}). Let
$$(g,h)\triangleq \int_{\Omega}\big(\partial_xg\partial_xh+|x|^{2\gamma}\partial_yg
\partial_yh\big)\,dxdy
\;\;\;\text{and}\;\;\;|g|_{V}\triangleq(g,g)^{\frac12},\; g,h\in
C_0^\infty(\Omega).$$ Set
$V=\overline{C_0^\infty(\Omega)}^{|\cdot|_V}$. Define a bilinear
form $a(\cdot, \cdot)$ over $V$ by
$$a(g,h)=-(g,h)\;\;\text{for all}\;\; g,h\in V,$$
and  an operator $A$ on   $X\triangleq L^2(\Omega)$  by
\begin{equation*}
\begin{cases}
D(A)=\big\{g\in V\,:\, \text{there is a constant}\; C\;\text{such that}\; |a(g,h)|\leq C \|h\|_{L^2(\Omega)}\;\,\text{for all}\; h\in V
\big\},\\
\lg Ag,h\rg_{L^2(\Omega)}=a(g,h)\;\;\text{for all}\;\;g\in D(A)\;\;\text{and}\;\;h\in V.
\end{cases}
\end{equation*}
Let  $U=L^2(\omega)$. Define  $B\in\mathcal L(U,X)$  by $Bf=f$ for
all $f\in U$. Then, $A$ is a self-adjoint operator and generates an
analytic semigroup in $X$; $B$ is an admissible control operator for
$\{e^{tA}; t\geq 0\}$ and $B^*:X\rightarrow U$ is given by
$B^*\varphi=\chi_\omega \varphi$ for all $\varphi\in X$; and
Equation~\eqref{405c9} can be rewritten as
\begin{equation*}
\frac{dz}{dt}=Az+Bf,\ t>0, \ z(0)=z_0.
\end{equation*}
Meanwhile, from Proposition 6 and from  the proof of Proposition 8
in \cite{BCG},  there is a constant $C=C(\Omega,\omega)>0$ such that
for any $L\in(0,1]$,
\begin{equation*}
\int_{\Omega}|\varphi(x,y,L)|^2\,dxdy
\leq e^{C{L^{-\frac{1+\gamma}{1-\gamma}}}}\int_{0}^L\!\int_\omega |\varphi(x,y,t)|^2\,dxdydt\;\;\mbox{for all}\;\; \varphi_0\in L^2(\Omega),
\end{equation*}
where $\varphi$ solves
\begin{equation*}
\begin{cases}
\varphi_t-\partial_x^2\varphi-|x|^{2\gamma}\partial_y^2\varphi=0,\;\;&\text{in}\;\;
\Omega\times(0,L),\\
\varphi=0,\;\;&\text{on}\;\;\partial\Omega\times(0,L),\\
\varphi(x,y,0)=\varphi_0,\;\;&\text{in}\;\;\Omega.
\end{cases}
\end{equation*}
In other words, $(A^*,B^*)$ satisfies  observability inequality
\begin{equation*}
\|e^{LA^*}\varphi_0\|_X^2\leq e^{C{L^{-\frac{1+\gamma}{1-\gamma}}}}\int_0^L
\|B^*e^{tA^*}\varphi_0\|_U^2\,dt\;\;\mbox{for all}\;\; \varphi_0\in X \;\;\mbox{and}\;\; L\in(0,1].
\end{equation*}
Therefore, one can apply Theorem~\ref{thejingjingjingcan1} to deduce the following corollary:
\begin{Corollary}
  Problem $(TP)^M_4$  holds the bang-bang property.
\end{Corollary}


\bigskip

\subsubsection{Parabolic equations with coefficients jumping  at an interface}
Let $\Omega$ be a smooth and bounded domain in $\mathbb R^n$
($n\geq2$) and  $\omega\subset\Omega$ be a nonempty open subset.
Regard $L^2(\omega)$  as a subspace of $L^2(\Omega)$ by extending
functions in $L^2(\omega)$ to be zero in $\Omega\setminus\omega$.
Define an operator $\mathbb L$ in $L^2(\Omega)$ by
 $$
 \mathbb L = \text{div}(a(x)\nabla),
 $$
 with
$$
D(\mathbb L)=\{u\in H_0^1(\Omega)\,:\, \text{div}(a(x)\nabla u)\in
L^2(\Omega)\},
$$
 where $a$ verifies
 $$
 0<a_1\leq a(x)\leq a_2<+\infty \;\;\mbox{over}\;\;
 \Omega.
 $$
 The coefficient
$a(\cdot)$ is further assumed smooth apart from across an interface
$\Gamma$, where it may jump. The interface $\Gamma$ is the boundary
of a smooth open subset of $\Omega$.

Consider the following time optimal
control problem:
$$(TP)_5^M:\;\;\;\;\quad\;T(M)\triangleq\inf_{f\in\mathcal U_M}\big\{t>0:
y(t;f)=0\big\},
$$
where
$$
\mathcal U_M\triangleq\big\{f\in L^\infty(\mathbb
R^+;L^2(\omega)):\, \| f(t)\|_{L^2(\omega)}\leq M,\, \text{a.e.}\,\
t>0 \big\},\;\;\mbox{with}\;\; M>0,
$$
 and $y(\cdot,f)$ is the solution to
\begin{equation*}
\begin{cases}
y_t-\mathbb Ly= f,\;\;&\text{in}\;\;\Omega\times\mathbb R^+,\\
y=0,\;\;&\text{on}\;\;\partial\Omega\times\mathbb R^+,\\
y(\cdot,0)=y_0,\;\;&\text{in}\;\;\Omega,
\end{cases}
\end{equation*}
with $y_0\in L^2(\Omega)\setminus\{0\}$.

 Let $\{\lambda_m\}_{m\geq1}$, sorted in an increasing sequence, and $\{e_m\}_{m\geq1}$ be the sets of the
  eigenvalues and of the associated $L^2(\Omega)$-normalized eigenfunctions of the operator $-\mathbb L$, respectively.
According to  \cite[Theorem 1.2]{RR}, there exists a constant
$N=N(\Omega,\omega)\geq1$ such that  the spectral inequality
\begin{equation}\label{sp}
\|g\|_{L^2(\Omega)}\leq Ne^{N\sqrt{\lambda_m}}\|\chi_\omega g\|_{L^2(\omega)},
\end{equation}
holds for all $m\in\mathbb N$ and every function $g\in \mathbb E_{\lambda_m}\triangleq \text{span}\{e_j\,:\,j\leq m\}$.

Let $X=L^2(\Omega)$, $U=L^2(\omega)$ and $A=\mathbb L$. Define
$B\in\mathcal L(U,X)$  by $Bf=f$ for all $f\in U$. From \eqref{sp},
it is easy to see that Hypothesis $(H)$ (in Theorem~\ref{speth})
holds in this case. Hence, we have the following consequence of
Theorem~\ref{speth}.

\begin{Corollary}
 Problem $(TP)^M_5$ holds the bang-bang property.
\end{Corollary}


\end{document}